\documentclass{amsart}
\usepackage{amsmath,amssymb,amsthm, mathrsfs,url}
\usepackage{boxedminipage}
\usepackage{graphicx}
\usepackage[margin=1in]{geometry}
\usepackage{mathpazo}
\usepackage{fancyhdr}
\usepackage[utf8]{inputenc}
\usepackage{color}
\usepackage{tikz}

\newtheorem{theorem}{Theorem}[section]
\newtheorem{lemma}[theorem]{Lemma}

\newtheorem{corollary}[theorem]{Corollary}
\newtheorem{proposition}[theorem]{Proposition}

\newtheorem{varexample}[theorem]{Example}

\newcommand{\F}{\mathbb{F}}
\newcommand{\K}{\mathbb{K}}
\newcommand{\Q}{\mathbb{Q}}
\newcommand{\HH}{\mathbb{H}}
\newcommand{\LL}{\mathbb{L}}

\theoremstyle{definition}

\title{Commuting graphs of $p$-adic matrices}
\author{Ralph Morrison}
\date{January 2023}

\begin{document}

\maketitle

\begin{abstract} We study the commuting graph of \(n\times n\) matrices over the field of \(p\)-adics \(\mathbb{Q}_p\), whose vertices are non-scalar \(n\times n\) matrices with entries in \(\mathbb{Q}_p\) and whose edges connect pairs of matrices that commute under matrix multiplication.  We prove that this graph is connected if and only if \(n\geq 3\), with \(n\) neither prime nor a power of \(p\).  We also prove that in the case of \(p=2\) and \(n=2q\) for \(q\) a prime with \(q\geq 7\), the commuting graph has the maximum possible diameter of \(6\); these are the first known such examples independent of the axiom of choice. We also find choices of $p$ and $n$ yielding diameter $4$ and diameter $5$ commuting graphs, and prove general bounds depending on $p$ and $n$.
\end{abstract}

\section{Introduction}

Let \(\F\) be a field, let \(n\geq 2\), and let \(M_n(\F)\) denote the set of \(n\times n\) matrices over \(\F\).  The \emph{commuting graph of \(M_n(\F)\)}, denoted \(\Gamma(\F,n)\), has as its vertices all nonscalar \(n\times n \) matrices with entries in \(\F\), with two vertices adjacent if and only if the corresponding matrices commute under matrix multiplication.  The \emph{commuting distance} between two nonscalar matrices \(A,B\in M_n(\F)\), denoted \(d(A,B)\), is their distance on \(\Gamma(\F,n)\).  That is, if
\[A=X_0-X_1-\cdots-X_{m-1}-X_m=B\]
is the shortest chain of nonscalar matrices connecting \(A\) and \(B\) where each pair of adjacent matrices commutes, then \(d(A,B)=m\); and if no such chain exists, then \(d(A,B)=\infty\).

A great deal of work has been done in studying the connectivity and diameter of \(\Gamma(\F,n)\), for various choices of \(n\) and \(\F\).  It turns out that for any field \(\F\), \(\Gamma(2,\F)\) is disconnected with cliques as its components \cite{disconnected_n=2}.  It was shown in \cite{when_connected} that \(\Gamma(\Q,n)\) is disconnected for all \(n\), while \(\Gamma(\F,n)\) is connected for \(n\geq 3\) if \(\F=\mathbb{R}\) or \(\F\) is algebraically closed.  Moreover, in the event that \(\Gamma(\F,n)\) is connected, its diameter is either \(4\), \(5\), or \(6\) \cite{diameter_6}.

It was proved in \cite{diameter_6} that over an algebraically closed field and with \(n\geq 3\), the commuting graph has diameter \(4\).  The same result was proved for \(\mathbb{R}\) in \cite{real1,real2,real3}.  There are also concrete examples of commuting graphs with diameter \(5\) \cite{diameter_5_equal}.  A longstanding conjecture was that the diameter is never equal to \(6\) \cite{diameter_6}.  This was refuted by a counterexample in \cite{counterexample_to_conjecture}, which used Zorn's lemma to construct a field giving rise to a commuting graph of diameter \(6\).

In this paper we study the commuting graph of \(n\times n\) matrices over \(\Q_p\), the field of \(p\)-adic numbers.  Our first result determines when \(\Gamma(\mathbb{Q}_p,n)\) is connected.  It follows readily from work in \cite{when_connected} combined with standard results on finite field extensions of \(\mathbb{Q}_p\).

\begin{theorem}\label{theorem:connectivity_qp} Let \(n\geq 3\) and let \(p\) be prime. The commuting graph \(\Gamma(\mathbb{Q}_p,n)\) is connected if and only if \(n\) is neither prime nor a power of \(p\).
\end{theorem}

%Matrices over the \(p\)-adics provide a fruitful ground for constructing examples of connected commuting graphs with high diameter.
Our next major result is Theorem \ref{theorem:sufficient_condition}, which provides a sufficient condition on \(n\) and \(p\) to conclude that \(\Gamma(\mathbb{Q}_p,n)\) has the maximum possible diameter of \(6\).  This theorem, and the proof thereof, closely follows the work of \cite{diameter_commuting}, where it was proved that for \(q\) prime, the disconnected graph \(\Gamma(\mathbb{Q},2q)\) consists of cliques plus one component with diameter equal to \(6\).  We use our result to prove the following.

\begin{theorem}\label{theorem:q2_14} Let \(q\geq 7\) be prime.  The commuting graph of \(2q\times 2q\) matrices over \(\Q_2\) has diameter \(6\).
\end{theorem}

This answers in the affirmative the question posed in \cite[\S 4]{real3} of whether a commuting graph of diameter \(6\) could be proved to exist without using the axiom of choice.  Indeed, that paper pointed to the \(p\)-adics as a natural place to search for such an example.

This large diameter is not universal for all connected \(p\)-adic commuting graphs. We present infinite families of examples with diameter equal to \(4\) (Proposition \ref{prop:diameter_four}), diameter at least \(5\) (Proposition \ref{prop:square_of_prime}), and diameter at most \(5\) (Proposition \ref{prop:sqrt_n} and Corollaries \ref{corollary:small_prime_factor} and \ref{corollary:congruence}), as well as an example with diameter exactly $5$ (Theorem \ref{theorem:all_diameters_possible}).

Our paper is organized as follows.  In Section \ref{section:connectivity} we prove Theorem \ref{theorem:connectivity_qp}.  In Section \ref{section:diameter} we provide our sufficient condition for the commuting graph over the \(p\)-adics to have diameter \(6\), and apply it to \(2q\times 2q\) matrices over \(\Q_2\) with \(q\geq 7\) a prime.  In Section \ref{section:possible_diameters} we present other bounds on diameter, and present directions for future work.

\noindent \textbf{Acknowledgements.} The author thanks Joachim K\"{o}nig for assistance in the proof of Lemma \ref{lemma:mathoverflow}.

\section{Connectivity of $p$-adic commuting graphs}
\label{section:connectivity}

We say that a field extension \(\mathbb{K}/\mathbb{F}\) is \emph{primitive} if there is no intermediate field between \(\mathbb{K}\) and \(\mathbb{F}\).  The connectivity of \(\Gamma(\F,n)\) can be reduced to a question of the existence or non-existence of certain primitive field extensions. 

\begin{theorem}\label{theorem:connected_by_fields}[\cite{when_connected}]  Let \(\mathbb{F}\) be a field and \(n\geq 3\).  The commuting graph \(\Gamma(\mathbb{F},n)\) is connected if and only if there does not exist a primitive field extension of \(\mathbb{F}\) with degree \(n\).
\end{theorem}

We recall some standard results on the structure of finite field extensions of the \(p\)-adic numbers \(\mathbb{Q}_p\); we refer the reader to \cite[Ch. II, \S 7]{neukirch} and  \cite[Ch. 14.3, Ch. 16]{hasse} for more details.

Let \(\LL \supset\K\supset\mathbb{Q}_p\) be finite field extensions, where \([\LL:\K]=n\), and let \(\pi\) be a uniformizer for \(\K\), with \(p^h\) the number of elements of the residue field of \(\K\).  We may decompose the extension \(\LL/\K\) as
\[\mathbb{K}\subseteq\mathbb{K}'\subseteq \mathbb{K}''\subseteq\mathbb{L},\]
where \(\mathbb{K}'\) is an unramified extension of \(\K\) of degree \(f\),  \(\mathbb{K}''\) is a tame, purely ramified extension of \(\mathbb{K}'\) of degree \(e\), and \([\mathbb{L}:\mathbb{K}'']=p^k\) for some \(k\).  We have that
\[\K'=\K(\zeta)\]
for some primitive \((p^{hf}-1)st\) root of unity \(\zeta\), and that the extension \(\mathbb{K}'/\mathbb{K}\) is Galois with cyclic Galois group. In fact, for every \(f\), \(\mathbb{K}\) has a unique unramified extension of degree \(f\), again Galois with cyclic Galois group.  We also have that up to conjugation,
\[\K''=\mathbb{K}(\zeta,(\pi \zeta^r)^{1/e}),\]
where \(0\leq r<\gcd(e,p^{hf}-1)\).  Moreover, \(\K''\) is a Galois extension of \(\mathbb{K}\) if and only if \(p^{hf}\equiv 1\mod e\) and \(r(p^h-1)\equiv 0\mod e\); and the extension is Abelian if and only if \(p^h-1\equiv 0\mod e\).

We now present the following lemma, which will allow us to apply Theorem \ref{theorem:connected_by_fields}.  We let \(C_n\) denote the cyclic group with \(n\) elements.

\begin{lemma}\label{lemma:qp_intermediate} Let \(p\) be prime and let \(n\geq 2\).  There exists a primitive field extension \(\mathbb{L}\) of  \(\mathbb{Q}_p\)  of degree \([\mathbb{L}:\mathbb{Q}_p]=n\) if and only if \(n\) is either prime or a power of \(p\).
\end{lemma}

\begin{proof}  First assume that \(n\) is either prime or a power of \(p\).  If \(n\) is prime, then any degree \(n\) extension of \(\mathbb{Q}_p\) is primitive by the Tower Law; and there do exist extensions of every possible degree of \(\Q_p\) (for instance, for every \(n\) there exists a unique unramified extension of \(\Q_p\) of degree \(n\)).  If \(n=p^k\) for some \(k\geq 2\), then there exists a primitive extension of \(\mathbb{Q}_p\) of degree \(n\); see for instance \cite{powers_of_p}.

%Now assume that \(n\) is neither prime nor a power of \(p\); write \(n=ab\) where \(a,b\geq 2\) are integers.  
Now assume that there exists a primitive field extension \(\mathbb{L}\) of \(\mathbb{Q}_p\) of degree \(n\).  Consider the usual field decomposition
\[\mathbb{Q}_p\subseteq\mathbb{K}'\subseteq \mathbb{K}''\subseteq\mathbb{L},\]
where \(\mathbb{K}'\) is an unramified extension of \(\mathbb{Q}_p\),  \(\mathbb{K}''\) is a tame, purely ramified extension of \(\mathbb{K}'\), and \([\mathbb{L}:\mathbb{K}'']=p^k\) for some \(k\).  Since \(\LL\) is a primitive extension, exactly one of these field extensions is nontrivial.
\begin{itemize}
\item If the first extension is nontrivial, then $\LL$ is an unramified extension of $\Q_p$. This means it is Galois with cyclic Galois group \(C_n\).  As there are no intermediate fields, by the Fundamental Theorem of Galois Theory there must be no proper nontrivial subgroups of  \(C_n\), meaning that \(n\) is prime.
\item If the second extension is nontrivial, then $\LL$ is a tame, totally ramified extension of $\Q_p$.   It follows that \(\mathbb{K}=\mathbb{Q}_p((up)^{1/n})\) for some unit \(u\) in the \(p\)-adic integers.  If \(n\) were composite, say with \(d|n\) where \(1<d<n\), then there would be an intermediate field \(\mathbb{Q}_p((up)^{1/d})\).  As no intermediate field exists, \(n\) must be prime.
 \item If the third extension is nontrivial, then $n=[\LL:\Q_p]=p^k$ for some \(k\).
 \end{itemize}

In all cases we have that \(n\) is either prime or a power of \(p\), completing the proof.
\end{proof}

Combining Theorem \ref{theorem:connected_by_fields} and Lemma \ref{lemma:qp_intermediate}, we immediately have Theorem \ref{theorem:connectivity_qp}.

We have one more lemma on $p$-adic fields to prove, which will be useful in Section \ref{section:possible_diameters}.

\begin{lemma}\label{lemma:mathoverflow}
Suppose that \(p\nmid n\), and that \(q\) is the smallest prime factor of \(n\).  Then for any finite extensions \(\LL\supset \K\supset \mathbb{Q}_p\) with \([\LL:\K]=n\), there exists an intermediate subextension of degree \(q\) over \(\K\).
\end{lemma}

\begin{proof}  We remark that \(\LL/\K\) must be tamely ramified, since \(p\nmid n\).  This lets us consider the decomposition
\[\K\subset \K'\subset \LL,\]
where \(\K'\) is unramfied over \(\K\) and \(\LL\) is totally ramified over \(\K'\).  Let's first argue that there exists some field between \(\K\) and \(\LL\) of some prime degree over \(\K\).  If the extension is not totally ramified, then since \(\K'/\K\) is unramified, it must be Galois and cyclic.  Thus between \(\K\) and \(\K'\) we must be able to find fields with degree equal to any given divisor of \([\K':\K]\), including any prime divisor.  Otherwise, the extension \(\LL/\K\) must be totally ramified, so that \(\LL=\K((\pi)^{1/n})\).  This contains intermediate subfields of the form \(\LL=\K((\pi)^{1/d})\) for every divisor \(d\) of \(n\), including prime divisors, giving us an intermediate extension of prime order.   By the existence of some extension of prime degree, we have handled the case that \(n\) is a power of \(q\).  For the remainder we may assume that \(n\) has at least one prime factor besides \(q\).

We next argue that the claim holds when \(n=q\ell\) is the product of two primes,  with \(q<\ell\).  We consider 
\[\K\subset \K'\subset \LL\]
as before.  Since \(\K'\) is Galois over \(\K\), if \(q\) divides \([\K',\K]\) we have a subextension of degree \(q\) and we are done.  Otherwise, we may write \(\K'=\K(\zeta)\) and \(\LL=\K(\zeta,(\pi\zeta^r)^{1/e})\), where \(\zeta\) is a primitive \((p^{hf}-1)st\) root of unity, \(u\) is a unit in the ring of integers of \(\K\), and \(0\leq r<\gcd(p^{hf}-1,e)\).  Consider the subfield \(\tilde{\K}=\K(\zeta,(\pi\zeta^r)^{1/q})\). 
If \(\gcd(p^{hf}-1,q)=1\), then \(\tilde{\mathbb{K}}\) contains \(\K(\pi^{1/q}\)), which has degree \(q\). And if \(q\) divides \(p^h-1\), then \(\tilde{\K}\) is an abelian Galois extension of \(\K\), also leading to the desired extension of degree \(q\).  Thus we are left to handle the case that \(q\) divides \(p^{fh}-1\) but not \(p^h-1\); i.e. that \(p^h\not\equiv 1\mod q\) and  \(p^{hf}=1\mod q\).  In this case we must have \(f=\ell\), so that the multiplicative order of \(p^h\) modulo \(q\)
 is \(\ell\).  But \(\ell>q\), so this is impossible.  So in all cases where \(n\) has two prime factors, we have the intermediate field.  

We now use induction on the number \(r\) of prime factors of \(n\), counted with multiplicity. We have already handled the case with \(r=2\); assume our claim is true for some fixed \(r\geq 2\), and consider \(n\) with \(r+1\) prime factors.  As argued before, there must be \emph{some} subextension \(\LL\supset \LL'\supset \K\) of prime degree \(\ell\) over \(\K\). If \(\ell=q\), we are done.  If not, then \(q\) is the smallest prime dividing \([\LL:\LL']=n/\ell\), so there exists a field extension \(\LL''\) of degree \(q\) over \(\LL'\), and thus of degree \(\ell q\) over \(\K\).  Applying our result from the \(r=2\) case to the extension \(\LL'/\K\) gives us the field extension of degree \(q\) over \(\K\).
\end{proof}

\section{A commuting graph with diameter \(6\)}
\label{section:diameter}

We recall several results from \cite{diameter_commuting} that will be used in our next theorem.

\begin{lemma}\label{lemma:1}[\cite{diameter_commuting}, Lemma 2.3]  Let \(\F\) be a field and let \(X\in M_n(\F)\) with \(\F[X]\) a field.  Then the rational form of \(X\) consists of copies of identical cells.
\end{lemma}

\begin{lemma}\label{lemma:2}[\cite{diameter_commuting}, Lemma 2.4] Let \(\F\) be a field with characteristic not equal to \(2\); let \(q\geq 7\) be prime; and let \(C\in M_q(\F)\) be a companion matrix of some irreducible monic separable polynomial \(m(x)\in \F[x]\) so that \(\K:=\F[C]\) is its splitting field.  Then there exists a matrix \(U\in M_q (\F)\) such that \(U(I+U)\) is invertible, \(U^q\in \F I\), \(U\K=\K U\), and \[M_q(\F)=\K +\K U+\K U^2+\cdots + \K U^{q-1}.\]
Moreover, the sum is direct as a sum of left \(\K\)-modules.
\end{lemma}

\begin{lemma}\label{lemma:3}[\cite{diameter_commuting}, Proposition 3.4]  Let \(\F\), \(\K\) and \(U\) be as in Lemma \ref{lemma:2}.  Let $S=\left(\begin{matrix} I& U \\ U&-U^3\end{matrix}\right)\in M_{2q}(\F)$.  If \(2\times 2\) block matrices \(F,G\in M_2(\K)\subset M_{2q}(\F)\) are both nonscalar, then \(F\) and \(S^{-1}GS\) do not commute.
\end{lemma}

We are now ready to state and prove our sufficient condition for a commuting graph to have diameter \(6\).

\begin{theorem}\label{theorem:sufficient_condition} Let \(p,q\) be prime with \(q\geq 7\).  Suppose that there exists a Galois extension \(\mathbb{L}\) of \(\Q_p\) with Galois group \(G\) such that
\begin{itemize}
\item \(G\) has a subgroup \(H\) and a normal subgroup \(K\);
\item \(H\leq K\), and \(H\) is not contained in any other proper subgroup of \(G\); and
\item $[G:K]=q$ and $[K:H]=2$.
\end{itemize}
Then the diameter of \(\Gamma(\Q_p,2q)\) is \(6\).
\end{theorem}

Our proof follows that of \cite[Theorem 4.2]{diameter_commuting}. As previously, let \(C_n\) denote the cyclic group of order \(n\), and let \(\mathcal{C}_{M_n(\F)}(A)\) denote the commutator of \(A\), i.e. the set of matrices in \(M_n(\F)\) that commute with \(A\).

\begin{proof} Since \(q\geq 7\), we know that \(2q\) is neither prime nor a power of \(p\), so \(\Gamma(\Q_p,2q)\) is connected by Theorem \ref{theorem:connectivity_qp}.  By \cite[Theorem 17]{diameter_6}, we know the diameter of \(\Gamma(\Q_p,2q)\) is at most \(6\), so it suffices to prove that there exists a pair of matrices with distance \(6\).

Following the Fundamental Theorem of Galois Theory, let \(\mathbb{H}\) be the fixed field of \(H\) and \(\mathbb{K}\) be the fixed field of \(K\). By our subgroup assumptions, we know that \(\mathbb{K}\) is the unique intermediate field between \(\mathbb{H}\) and \(\Q_p\).  Moreover, since \(K\) is normal in \(G\), we know that \(\mathbb{K}\) is a Galois extension of \(\Q_p\); and the Galois group of \(\K\) over \(\Q_p\) has order \([G:K]=q\), and thus must be \(C_q\).

Since \(\Q_p\) has characteristic \(0\), every finite extension is a simple extension by the primitive element theorem. This lets us write \(\HH=\mathbb{Q}_p(\alpha)\) for some \(\alpha\in\mathbb{Q}_p\).  Let \(m_\alpha(x)\in\mathbb{Q}_p[x]\) be the minimal polynomial of \(\alpha\) over $\mathbb{Q}_p$, and let \(A\in M_{2q}(\Q_p)\) be the companion matrix of \(m_\alpha(x)\).  Since \(m_\alpha(x)\) is irreducible over $\mathbb{Q}_p$, we know \(\mathbb{Q}_p[A]\) is a field, and it follows from the Cayley-Hamilton Theorem that \(\mathbb{Q}_p[A]\) is isomorphic to \(\HH\).

Since \(m_\alpha(x)\) is the minimal polynomial of a matrix, namely \(A\), we know it must be separable, so it has no repeated zeros in the algebraic closure \(\overline{\mathbb{Q}}_p\).  Thus \(A\) is nonderogatory, implying by \cite[\S 2.08]{derogatory} that
\[\mathcal{C}_{M_{2q}(\Q_p)}(A)=\Q_p[A].\]
Since \(\Q_p[A]\) is a finite field extension of \(\Q_p\), every element \(X\in \mathcal{C}_{M_{2q}(\Q_p)}(A)\) is algebraic, so \(\Q_p[X]\) is a subfield of \(\Q_p[A]\) containing \(\Q_p\).  By the isomorphism with \(\HH\), there are three such fields:  \(\Q_p[A]\), \(\Q_p\), and the unique intermediate field. Moreover, by the primitive element theorem, any subfield of \(\Q_p[A]\) is of the form \(\Q_p[X]\).  Choose \(X\in \mathcal{C}_{M_{2q}(\Q_p)}(A)\) so that \(\Q_p[X]\) is the intermediate field, with \([\Q_p[X]:\Q_p]=q\).
By Lemma \ref{lemma:2}, the rational form of \(X\) consists of two identical \(q\times q\) cells.  We may then write \(X\) in its rational form as
\[X=T^{-1} (C \oplus C)T,\]
where \(C\in M_2(\Q_p)\).

Note that \(C\) is the companion matrix of the minimal polynomial of \(X\), and that \(\Q_p[X]\cong \Q_p[C]\subset M_q(\Q_p)\).    Since \(\K\) is Galois and \([\K:\Q_p]=q\) is the size of the matrix \(C\), we know that \(\K\) must be the splitting field for the minimal polynomial of \(C\).  This lets us apply Lemma \ref{lemma:2} with \(\F=\Q_p\) and our \(C\).

From here, we may let \(S\) be as in Lemma \ref{lemma:3}.  We will show that the distance between the matrices \(A_1=TAT^{-1}\) and \(B_1=S^{-1}A_1S\) is at least \(6\). Note that both \(A_1\) and \(B_1\) are nonderogatory.  Consider the shortest path between \(A_1\) and \(B_1\):  \[A_1=X_0\-- X_1\-- X_2\-- \cdots \-- X_{m-2}\-- X_{m-1}\-- X_m=B_1.\]
Since \(X_1\in\mathcal{C}_{M_{2q}(\Q_p)}(A_1)=\Q_p[A_1]=T\Q_p[A]T^{-1}\) is nonscalar, either \(\Q_p[X_1]=\Q_p[A_1]\) or \(\Q_p[X_1]=T\Q_p[X]T^{-1}=\Q_p[C\oplus C]\).  In the first case, \(\mathcal{C}_{M_{2q}(\Q_p)}(X_1)=\mathcal{C}_{M_{2q}(\Q_p)}(A_1)\); then \(X_2\) commutes with \(A_1\), meaning our path was not the shortest.  Thus we must be in the second case, so we have \(\mathcal{C}_{M_{2q}(\Q_p)}(X_1)=\mathcal{C}_{M_{2q}(\Q_p)}(C\oplus C)\).  Similarly, without loss of generality we have \(X_{m-1}=S^{-1}(C\oplus C)S\in \Q_p[B_1]\).  By Lemma \ref{lemma:3}, the matrices \(X_1\) and \(X_{m-1}\) do not commute, meaning \(d(A_1,B_1)\) is at least \(4\).  By Lemma \ref{lemma:3}, since \(X_2\in \mathcal{C}_{M_{2q}(\mathbb{Q}_p)}(X_1)\) we have \(X_2\in M_2(\Q_p[C])\); and similarly \(X_{m-2}\in S^{-1}M_2(\Q_p[C])S\).  Moreover, \(X_1\in M_2(\Q_p[C])\) and \(X_{m-1}\in S^{-1}M_2(\Q_p[C]) S\).  Thus by Lemma \ref{lemma:3} neither matrix from \(\{X_1,X_2\}\) commutes with either matrix from \(\{X_{m-2},X_{m-1}\}\), meaning that \(d(A_1,B_1)>5\).  This completes the proof.
\end{proof}

As a concrete application of this result, consider the wreath product \(C_2\wr C_7\), defined as the semidirect product \(C_2^7\rtimes C_7\) where \(C_7\) acts on \(C_2^7\) by cyclic rotation.  As shown in \cite[Lemma 2.5]{diameter_commuting}, \(C_2\wr C_7\) satisfies the group-theoretic hypotheses of Theorem \ref{theorem:sufficient_condition}.  Moreover, there exists a Galois extension \(\LL\) of \(\Q_2\) with Galois group \(C_2\wr C_7\), namely the Galois closure of \(\Q_2(\alpha)\) where \(\alpha\) is a root of
\[x^{14}-x^{12}+2x^{11}+2x^{10}+2x^4+2x^3+1;\]
see \cite[$p$-adic field 2.14.14.13]{lmfdb}.  Thus we may apply Theorem \ref{theorem:sufficient_condition} with \(p=2\) and \(q=7\) to conclude that \(\Gamma(\Q_2,14)\) has diameter \(6\).

We now prove that this example generalizes when we replace \(7\) with any prime \(q\geq 7\).

\begin{proof}[Proof of Theorem \ref{theorem:q2_14}] Let \(q\geq 7\) be prime. First we will prove that there exists a degree \(2q\) extension \(\HH\) of \(\Q_2\) containing precisely one intermediate subfield, which has degree \(q\). In particular, \(\HH\) will have residue degree \(f=q\) and ramification index \(e=2\). We can count the number of extensions of \(\Q_2\) with this structure using Krasner's lemma (proved in \cite{krasner}; see also \cite{enumeration}), and find that there are \(2^{q+2}-2\) such extensions. Every one of these fields will contain the unique unramified extension \(\K/\Q_2\) of degree \(q\). Note that none of the degree \(2q\) extensions can contain another subfield of degree \(q\):  if they did they would contain the compositum of the two distinct degree \(q\) subfields, which would have degree \(q^2>2q\).  Similarly, none can contain more than one quadratic subfield:  if one did, it would have to contain the degree \(4\) compositum of the two fields, but \(4\nmid 2q\).  Finally, if \(\HH\) does contain a quadratic subfield, by the Tower Law it would have to be equal to the compositum of \(\K\) and that quadratic subfield, meaning that no two \(\HH\) may contain the same quadratic subfield.  Since there are \(2^{q+2}-2\) choices for \(\K\) and only \(7\) quadratic extensions of \(\Q_2\), we conclude that at least one choice of \(\HH\) has no quadratic subfield, and thus no subfields besides \(\K\).

Now let \(\LL\) be the Galois closure of \(\HH\) over \(\mathbb{Q}_2\), with Galois group \(G\).  Since \(\K/\mathbb{Q}_p\) is a Galois extension, it corresponds to a normal subgroup \(K\) of index \(q\) in \(G\).  Since \(\HH\) is a degree \(2\) extension of \(\K\) containing no other subfields besides \(\Q_2\), it corresponds to a subgroup \(H\leq K\leq G\) with \(H\) of index \(2\) in \(K\), such that no other proper subgroup of \(G\) contains \(H\).  Thus the hypotheses of Theorem \ref{theorem:sufficient_condition} are met, and we can conclude that the diameter of \(\Gamma(\Q_2,2q)\) is \(6\).
\end{proof}

\section{Possible diameters over \(\Q_p\)}\label{section:possible_diameters}

Having established that \(\Gamma(\Q_2,2q)\) has diameter \(6\) for $q\geq 7$ a prime, we turn to the question of what other diameters are possible for connected \(\Gamma(\Q_p,n)\).  The first two results are quick applications of known theorems.

\begin{proposition}\label{prop:diameter_four} For \(p\neq 2\), the diameter of \(\Gamma(\Q_p,4)\) is \(4\).
\end{proposition}

\begin{proof} By Theorem \ref{theorem:connectivity_qp}, we  know the graph is connected since \(p\neq 2\).  By \cite{diameter_finite_fields}, we have that  \(\Gamma(\F,4)\) over any field \(\F\) has diameter \(4\) as long as it is connected.
\end{proof}

%Our next two results give a lower bound of \(5\) and an upper bound of \(5\), respectively, on the diameter of \(\Gamma(\Q_p,n)\) for certain choices of \(p\) and \(n\).

\begin{proposition}\label{prop:square_of_prime}  Let \(p,q\) be prime with \(p\neq q\) and \( q\geq 3\), and let \(n=q^2\).  The diameter of \(\Gamma(\Q_p,n)\) is at least \(5\).
\end{proposition}

\begin{proof}  By \cite[Theorem 3.2]{diameter_5_at_least}, this result holds for any infinite field \(\F\) such that \(\Gamma(\F,n)\) is connected, and such that \(\F\) admits a Galois extension of degree \(n\) with a cyclic Galois group.  Since \(n\) is neither prime nor a power of \(p\), the graph is connected; and \(\Q_p\) admits such an extension, namely the unique unramified extension of degree \(n\).
\end{proof}

We present one more lower bound on the diameter of $\Gamma(\mathbb{Q}_p,n)$.  Let $\mathbb{F}_p$ denote the finite field with $p$ elements.

\begin{proposition}\label{prop:lower_bound_mod_p}
        Let $p$ be a prime, and let $\mathbb{F}_p$ be the field with $p$ elements.  Then 
    \[\textrm{diam}(\Gamma(\mathbb{Q}_p,n))\geq \textrm{diam}(\Gamma(\mathbb{F}_p,n))).\]
\end{proposition}

It is worth remarking that this inequality can be strict.  By \cite[Theorem 3.3]{diameter_finite_fields}, if $n$ is neither prime nor the square of a prime, then $\Gamma(\mathbb{F}_p,n)$ has diameter at most $5$.  So, $\Gamma(\mathbb{F}_2,14)$ has diameter at most $5$, while $\Gamma(\mathbb{Q}_2,14)$ has diameter $6$.

\begin{proof}
    Let $d=\textrm{diam}(\Gamma(\mathbb{Q}_p,n))$.  Suppose for the sake of contradiction that $\textrm{diam}(\Gamma(\mathbb{Q}_p,n))<\textrm{diam}(\Gamma(\mathbb{F}_p,n)))$ (which rules out $d=\infty$).  Let $\overline{A}$ and $\overline{B}$ be $n\times n$ non-scalar matrices with entries in $\mathbb{F}_p$ such that $d(\overline{A},\overline{B})>d$ when measured on $\Gamma(\mathbb{F}_p,n))$.  Identifying the elements of $\mathbb{F}_p$ with $\{0,1,\ldots,p-1\}$,  we may set $A,B\in M_n(\mathbb{Q}_p)$ to have the same entries as $\overline{A}$ and $\overline{B}$; by construction these are non-scalar.  Since $d=\textrm{diam}(\Gamma(\mathbb{Q}_p,n))$, there exists a chain of non-scalar matrices
    \[A=X_0-X_1-\cdots-X_{d-1}-X_d=B\]
    where each adjacent pair of matrices commutes\footnote{If $d(A,B)<d$, we can still get a chain of length exactly $d$ by adding repeated copies of matrices.}.  Since $X_1$ is a non-zero matrix, we can redefine $X_1:=p^kX_1$ for some integer $k$ so that the minimum valuation of any entry is $0$, yielding matrices whose entries are $p$-adic integers
    \[a_0+a_1p+a_2p^2+\cdots,\]
    at least one of which has $a_0\neq 0$.  Let $\overline{X_1}$ denote the reduction of $X_1$ modulo $p$.

If $\overline{X_1}$ is non-scalar, leave $X_1$ as is.  If $\overline{X_1}$ is scalar, then each off-diagonal entry $a_0+a_1p+a_2p^2+\cdots$ has $a_0=0$, and each diagonal-entry has leading term $a_0=\lambda$ for some fixed non-zero $\lambda$.  In this case, redefine $X_1:=\frac{1}{p^\ell}(X_1-\lambda I)$, where $\ell\geq 1$ is chosen so that the minimum valuation over all entries of our new $X_1$ is $0$.  If $\overline{X_1}$ is non-scalar, we leave $X_1$ as is; otherwise we repeat this process.

Eventually, this process will terminate, either when a non-zero off-diagonal entry is scaled down to have valuation $0$, or when two non-equal diagonal entries have enough terms cancelled that they have different leading terms.  Thus we end with a choice of $X_1$ such that $\overline{X_1}$ is non-scalar, and since our alterations do not affect commutativity, we still have that $X_1$ commutes with $X_0$ and with $X_2$.

Perform the same process for $X_2,\ldots,X_{d-1}$.  Thus we have a commuting chain
\[A=X_0-X_1-\cdots-X_{d-1}-X_d=B\]
such that $\overline{X_i}$ is non-scalar for all $i$.

Note that if two matrices $C$ and $D$ with $p$-adic integer entries commute, so too must $\overline{C}$ and $\overline{D}$ over $\mathbb{F}_p$, as reduction modulo $p$ is a ring homomorphism from the $p$-adic integers to $\mathbb{F}_p$, and so preserves the addition and multiplication operations that dictate commutativity.  It follows that 
\[\overline{A}=\overline{X_0}-\overline{X_1}-\cdots-\overline{X_{d-1}}-\overline{X_d}=\overline{B}\]
is a commuting chain of non-scalar matrices over $\mathbb{F}_p$.  Thus we have $d(\overline{A},\overline{B})\leq d$, a contradiction.

\end{proof}

We now turn to upper bounds. The following proposition, whose proof closely follows that of Theorem 3.3 from \cite{diameter_finite_fields}, will allow us to deduce that many commuting graphs over the \(p\)-adics have diameter at most \(5\).

\begin{proposition}\label{prop:sqrt_n} Let \(p\) and \(n\) be such that \(\Gamma(\Q_p,n)\) is connected and every extension \(\K/\Q_p\) of degree \(n\) has a subextension of degree strictly less than \(\sqrt{n}\).   Then the diameter of \(\Gamma(\Q_p,n)\) is at most \(5\).
\end{proposition}

\begin{proof} Let \(A,B\in M_n(\Q_p)\).  Assume first that \(A\) and \(B\) are both companion matrices to irreducible polynomials \(m_A\) and \(m_B\), both of degree \(n\).  Thus \(\Q_p[A]\) is a field extension of \(\Q_p\) of degree \(n\) \cite[Lemma 2.1]{diameter_finite_fields}. As such, it contains an intermediate field extension of degree \(d_1\) with  \(1<d_1<\sqrt{n}\). By the primitive element theorem, we may write this field extension as \(\Q_p[X]\) for some nonscalar \(X\in M_n(\Q_p)\).  Note that no polynomial in \(X\) can be an idempotent or a nilpotent, as this is impossible in the field \(\Q[X]\) since \(X\) is nonscalar.   It follows by \cite[Lemma 2.2]{diameter_finite_fields} that the minimal polynomial of \(X\) has degree \(d_1\), and the rational form of \(X\) consists of identical cells \(C_1\in M_d(\F)\), where \(C_1\) is the companion matrix of some monic polynomial of degree \(d_1\). This lets us write \[X=T(C_1\oplus\cdots\oplus C_1)T^{-1}.\]  An identical argument gives us that \(\Q_p[B]\) is a field with an intermediate field extension \(\Q_p[Y]\) of degree \(d_2\) with \(1<d_2<\sqrt{n}\), and we can write
\[Y=S(C_2\oplus\cdots\oplus C_2)S^{-1}\] for some \(D_1\in M_d(\F)\), where \(C_2\) is the companion matrix of some monic polynomial of degree \(d_2\).  The fact that \(d_1d_2<n\) allows us to apply \cite[Lemma 3.1]{diameter_finite_fields} to conclude that \(d(X,Y)\leq 3\).  Since \(A\) commutes with \(X\) and \(Y\) commutes with \(B\), we have \(d(A,B)\leq 5\).

Now assume that at least one of \(\Q_p[A]\) and \(\Q_p[B]\) is not a field.  We may apply \cite[Remark 3.4]{diameter_finite_fields} to deduce that \(d(A,B)\leq 5\).
\end{proof}

We use this proposition to prove two corollaries.

\begin{corollary}\label{corollary:small_prime_factor} Assume that the largest prime factor of \(n\) is strictly less than \(\sqrt{n}\), and that \(n\) is not a power of \(p\).  Then the diameter of \(\Gamma(\Q_p,n)\) is at most \(5\).
\end{corollary}

\begin{proof}  Given \(\K\) of degree \(n\) over \(\Q_p\), let  \(\K''/\Q_p\) be the maximal tame subextension, which must have degree \(m\) divisible by some prime factor \(q\) of \(\sqrt{n}\) with \(q\neq p\).  If \(\K''\) is totally ramified, then \(\K''=\Q_p((up)^{1/m})\) for some unit \(u\) in the \(p\)-adic integers.  In this case there is an intermediate field extension of degree \(q\), namely \(\Q_p((up)^{1/q})\).  If \(\K''\) is not totally ramified, there is an intermediate unramified extension \(\K'\), whose degree \(m'\) divides \(m\).  Since every unramified extension is Galois with cyclic Galois group, there is a subextension of \(\K'\) with degree equal to some prime factor of \(m\), which is strictly less than \(\sqrt{n}\).  In summary, for every extension \(\K/\Q_p\) of degree \(n\), there exists an intermediate extension of degree \(d<\sqrt{n}\).  This allows us to apply Proposition \ref{prop:sqrt_n} to complete the proof.
\end{proof}

\begin{corollary}\label{corollary:congruence}  Assume that \(n\) is an integer with a prime factor \(q\) distinct from \(p\), such that  \(q^2<n\) and at least one of the following three conditions holds:
\begin{itemize}
\item[(a)] for any \(f\) dividing \(n\) with \(q\nmid f\), we have \(q\nmid p^f-1\).
\item[(b)] \(p\equiv 1\mod q\).
\item[(c)] \(p\nmid n\).
\end{itemize}
Then the diameter of \(\Gamma(\Q_p,n)\) is at most \(5\).  
\end{corollary}

For instance, using condition (b) and choosing $q=2$, the diameter of \(\Gamma(\mathbb{Q}_p,n)\) is at most \(5\) if \(n\) is even and \(p\neq 2\). %similarly, the diameter of \(\Gamma(\mathbb{Q}_2,n)\) is at most \(5\) if \(n\) is odd and neither prime nor the square of a prime.  For an example of applying condition (a),   
%Maybe worth mentioning a time it doesn't apply:  take \(n=21\) and \(p=2\).  That way \(q=3\).  Note that \(3\) divides \(2^7-1\), so condition (a) doesn't hold, and \(2\) is not \(1\mod 3\), so (b) doesn't hold.  I think the theorem does still hold for these parameters--i.e. you always get that degree \(3\) subextension.  But the argument we have doesn't prove it yet.  I guess i this case it happens to be Galois, but even with Galois I don't know how you'd immediately know it works. 
To see that this does not cover all cases covered by Corollary \ref{corollary:small_prime_factor}, consider  \(p=2\) and \(n=12\).  To try to apply Corollary \ref{corollary:congruence}, we would have to choose \(q=3\), but \(2^4-1\) is divisible by \(3\), and $2\not\equiv 1\mod 3$.

\begin{proof}  First we note that the assumptions on \(n\) and \(q\) imply that \(n\) is neither prime nor a power of \(p\), so the commuting graph is connected.  Let \(\K/\Q_p\) be a extesion of degree \(n\).  It will suffice to show that there is a subextension of degree \(q\), since \(q<\sqrt{n}\) will allow us to apply Proposition \ref{prop:sqrt_n}.

Let \(\Q_p\subset\K'\subset \K''\subset \K \) be the decomposition of the extension \(\K/\Q_p\), where \(\K'\) is the maximal unramified subextension and \(\K''\) is the maximal tamely ramified subextension.  If \(q\) divides \([\K':\mathbb{Q}_p]\), we are done as this is a Galois extension with cyclic Galois group of order divisible by \(q\).  Thus we may assume that \(q\) does not divide \([\K':\mathbb{Q}_p]\), and since \([\K:\K'']\) is a power of \(p\) we know that \(q\) divides \([\K'':\K']\)

Let \(f=[\K':\mathbb{Q}_p]\) and \(e=[\K'':\K']\).  We have that
\[\K''=\Q_p(\zeta,(up\zeta^s)^{1/e})\]
where \(\zeta\) is a primitive \((p^f-1)th\) root of unity, \(u\) is a unit in the \(p\)-adic integers, and \(0\leq s<\gcd(e,p^f-1)\).  Consider the subfield 
\[\LL=\Q_p(\zeta,(up\zeta^s)^{1/q}).\]
Note that this extension is tamely ramified.  If condition (a) is satisfied, we may take \(0\leq s<\gcd(q,p^f-1)=1\):
\[\LL=\Q_p(\zeta,(up)^{1/q}).\]
This contains the subfield
\[\Q_p((up)^{1/q}),\]
which has degree at least \(q\) by the Tower Law, and at most \(q\) since \((up)^{1/q}\) is a root of the polynomial \(x^q-up\in \Q_p[x]\).  Thus we have an intermediate field of degree \(q\).

If on the other hand condition (b) is satisfied, by \cite[Ch. 16]{hasse} we know that \(\LL\) is an abelian Galois extension of \(\Q_p\).  That Galois group has order \(qf\), and since it is Abelian it has a subgroup of order \(f\) and index \(q\), which corresponds to an intermediate field extension of degree \(q\). 

Finally, if condition (c) holds, we may without loss of generality assume that \(q\) is the smallest prime factor of \(n\).  This lets us apply Lemma \ref{lemma:mathoverflow} to deduce that every extension of degree \(n\) over \(\mathbb{Q}_p\) has an extension of degree \(q\), as desired.
%we may without loss of generality assume that \(q\) is the smallest prime factor of \(n\).  It will suffice to prove the following:  if \(K/\mathbb{Q}_p\) has degree \(n\) not divisible by \(p\), where \(q<\sqrt{n}\) is the smallest prime factor of \(n\), then there is an intermediate field extension of degree \(q\) over \(\mathbb{Q}_p\).  By the previous arguments we know that there is an intermediate field extension \(\mathbb{Q}_p\subset K'\subset K\) of prime degree \(\ell\).  If \(\ell=q\), we are done.  Otherwise, 
\end{proof}

%Quick remark:  the only way neither (a) nor (b) holds is if \(q\) divides \(p^{f-1}+\cdots+p+1\).  (By the way, you can argue the extension is not Galois if you need to, in that case.)

Now armed with these results, we can prove what diameters are possible for $\Gamma(\mathbb{Q}_p,n)$ as $p$ and $n$ vary.

\begin{theorem}\label{theorem:all_diameters_possible}
    There exist choices of $p$ and $n$ such that the diameter of
$\Gamma(\mathbb{Q}_p,n)$ is $4$, $5$, $6$, and $\infty$.
\end{theorem}

\begin{proof}
    Diameter $4$ is achieved by $(p,n)=(p,4)$ for any prime $p\neq 2$ by Proposition \ref{prop:diameter_four}.  Diameter $6$ is achieved by $(p,n)=(2,2q)$ for any prime $q\geq 7$ by Theorem \ref{theorem:q2_14}.  Diameter $\infty$ is achieved by $(p,n)$ whenever $n$ is prime or a power of $p$ by Theorem \ref{theorem:connectivity_qp}.

    We claim that diameter $5$ is achieved by $(p,n)=(2,15)$.  We know the diameter of $\Gamma(\mathbb{Q}_2,15)$ is at most $5$ by Corollary \ref{corollary:congruence} with $q=3$ by condition (c).  It was show in \cite[\S 4]{diameter_5_equal} that the diameter of $\Gamma(\mathbb{F}_2,15)$ equals $5$, which by Proposition \ref{prop:lower_bound_mod_p} implies that the diameter of $\Gamma(\mathbb{Q}_2,15)$ is at least $5$.  This yields a diameter of exactly $5$.
\end{proof}

In many cases, we now have partial or definitive knowledge about the diameter of \(\Gamma(\mathbb{Q}_p,n)\).  We know precisely when it is finite; we know some instances when it is equal to \(4\), to $5$, and to \(6\); and we know many instances where it is at least \(5\) or at most \(5\).  There is still much work to do; for instance, we do not know if there are infinitely many choices of \(p\) and \(n\) that yield a diameter of $5$.  We remark that it was conjectured in \cite{real3} that if $\mathbb{F}$ is a finite field and $n$ is an odd composite number, then $\Gamma(\mathbb{F},n)$ has diameter equal to $5$.  Combined with our Proposition \ref{prop:lower_bound_mod_p} and Corollary \ref{corollary:congruence}, such a result would imply that $\textrm{diam}(\Gamma(\mathbb{Q}_p,n))=5$ as long as $n$ is an odd composite that is not the square with a prime, with $p\nmid n$.

Since there are only finitely many extensions of \(\Q_p\) of fixed degree, one can in principal check all degree \(n\) extensions for intermediate subfields of size less than \(\sqrt{n}\) to determine whether Proposition \ref{prop:sqrt_n} can be applied to bound diameter by $5$.  Using \cite{lmfdb}, we unfortunately find examples of $(p,n)$ where Proposition \ref{prop:sqrt_n} does not apply even when $n$ has a prime divisor smaller than $\sqrt{n}$, namely \((2,6)\) (\cite[$p$-adic field 2.6.6.4]{lmfdb}), \((2,10)\) (\cite[$p$-adic field 2.10.10.15]{lmfdb}), and \((3,15)\) (\cite[p-adic field 3.15.15.51]{lmfdb}).  These remain possible candidates to yield a diameter of \(6\).

%THINGS TO TRY FROM HERE:  could we lower bound \(\Gamma(\Q_3,8)\) by \(5\)? or upper bound \(\Gamma(\mathbb{Q}_2,9)\) by \(5\)?

\begin{table}[hbt]
\begin{tabular}{c|c|c|c|c|c|c|c|c|c|}
\cline{2-10}
 & $\mathbb{Q}_2$ & $\mathbb{Q}_3$ & $\mathbb{Q}_5$ & $\mathbb{Q}_7$ & $\mathbb{Q}_{11}$ & $\mathbb{Q}_{13}$ & $\mathbb{Q}_{17}$ & $\mathbb{Q}_{19}$ & $\mathbb{Q}_{23}$ \\ \hline
\multicolumn{1}{|c|}{$n=4$} & X & $4$ & $4$ & $4$ & $4$ & $4$ & $4$ & $4$ & $4$ \\ \hline
\multicolumn{1}{|c|}{$n=6$} & ? & $\leq 5$ & $\leq 5$ & $\leq 5$ & $\leq 5$ & $\leq 5$ & $\leq 5$ & $\leq 5$ & $\leq 5$ \\ \hline
\multicolumn{1}{|c|}{$n=8$} & X & $\leq 5$ & $\leq 5$ & $\leq 5$ & $\leq 5$ & $\leq 5$ & $\leq 5$ & $\leq 5$ & $\leq 5$ \\ \hline
\multicolumn{1}{|c|}{$n=9$} & $\geq 5$ & X & $\geq 5$ & $\geq 5$ & $\geq 5$ & $\geq 5$ & $\geq 5$ & $\geq 5$ & $\geq 5$ \\ \hline
\multicolumn{1}{|c|}{$n=10$} & ? & $\leq 5$ & $\leq 5$ & $\leq 5$ & $\leq 5$ & $\leq 5$ & $\leq 5$ & $\leq 5$ & $\leq 5$ \\ \hline
\multicolumn{1}{|c|}{$n=12$} & $\leq 5$ & $\leq 5$ & $\leq 5$ & $\leq 5$ & $\leq 5$ & $\leq 5$ & $\leq 5$ & $\leq 5$ & $\leq 5$ \\ \hline
\multicolumn{1}{|c|}{$n=14$} & 6 & $\leq 5$ & $\leq 5$ & $\leq 5$ & $\leq 5$ & $\leq 5$ & $\leq 5$ & $\leq 5$ & $\leq 5$ \\ \hline
\multicolumn{1}{|c|}{$n=15$} & $5$ & ? & $\leq 5$ & $\leq 5$ & $\leq 5$ & $\leq 5$ & $\leq 5$ & $\leq 5$ & $\leq 5$ \\ \hline
\multicolumn{1}{|c|}{$n=16$} & X & $\leq 5$ & $\leq 5$ & $\leq 5$ & $\leq 5$ & $\leq 5$ & $\leq 5$ & $\leq 5$ & $\leq 5$ \\ \hline
\multicolumn{1}{|c|}{$n=18$} & $\leq 5$ & $\leq 5$ & $\leq 5$ & $\leq 5$ & $\leq 5$ & $\leq 5$ & $\leq 5$ & $\leq 5$ & $\leq 5$ \\ \hline
\end{tabular}
\caption{Some known results on the diameter of the commuting graph \(\Gamma(\mathbb{Q}_p,n)\)}\label{table:summary}
\end{table}

In Table \ref{table:summary} we summarize our known bounds on the diameter of \(\Gamma(\mathbb{Q}_p,n)\) for \(n\leq 18\) and \(p\leq 23\).  We omit \(n\) prime, as this yields disconnected graphs for all \(p\).  An X indicates that the commuting graph is disconnected, and a ? that none of our results give us information beyond the universal bounds of at least \(4\) and at most \(6\). The results yielding the table's bounds are as follows:
\begin{itemize}
\item The disconnected graphs follow from Theorem \ref{theorem:connectivity_qp}.
\item The case \(n=14\), \(p=2\) follows from Theorem \ref{theorem:q2_14}.
\item  The case \(n=4\), \(p\geq 3\) follows from Proposition \ref{prop:diameter_four}.
\item  The case of \(n\geq 6\) even, \(p\geq 3\) follows from Corollary \ref{corollary:congruence} using condition (b).
\item The case of \(n=9\), \(p\neq 3\) follows from Proposition \ref{prop:square_of_prime}.
\item The case of \(n=15\), \(p=2\) was argued in the proof of Theorem \ref{theorem:all_diameters_possible}.
\item The case of \(n=15\), \(p=5\) follows from Corollary \ref{corollary:congruence} using condition (a).
\item  The case of \(n=15\) and \(p\notin \{3,5\}\) follows from Corollary \ref{corollary:congruence} using condition (c). 
\item The case of \(n\in\{12,18\}\), \(p=2\) follows from Corollary \ref{corollary:small_prime_factor}.
\end{itemize}
Characterizing precisely when diameters of \(4\), \(5\), and \(6\) occur would be an interesting direction for future research.

In the case of \(\Gamma(\mathbb{Q}_p,n)\) with \(n\geq 3\) prime or a power of \(p\), although the graph is disconnected, we can still ask for the diameters of its components.  By \cite[Lemma 4.1]{diameter_commuting}, all but one component is a clique, with the remaining component having diameter \(4\), \(5\), or \(6\).  An argument identical to that of \cite[Example 4.3]{diameter_commuting} shows that if \(n\) is prime, then the non-clique component of \(\Gamma(\mathbb{Q}_p,n)\) has diameter \(4\).  In the case that \(n=p^2\), \cite[Theorem 3.2]{diameter_5_equal} implies that the diameter of the non-clique is at least \(5\).  In the case that \(n=p^k\) for some \(k\geq 3\), we do not know of any results that narrow down  the diameter of the non-clique.

\bibliographystyle{alpha}

\end{document}